\documentclass[12 pt]{article}%
\usepackage{amsmath, amsfonts, amsthm, color,latexsym}
\usepackage{amsmath, ulem}
\usepackage{amsfonts}
\usepackage{amssymb}
\usepackage{color, soul}
\usepackage{xcolor}
\usepackage[all]{xy}
\usepackage{graphicx}%
\usepackage{hyperref}
\providecommand{\U}[1]{\protect\rule{.1in}{.1in}}
\allowdisplaybreaks[4]
\newtheorem{theorem}{Theorem}[section]
\newtheorem{proposition}[theorem]{Proposition}
\newtheorem{corollary}[theorem]{Corollary}
\newtheorem{example}[theorem]{Example}
\newtheorem{examples}[theorem]{Examples}

\newtheorem{remarks}[theorem]{Remarks}
\newtheorem{lemma}[theorem]{Lemma}
\newtheorem{final remark}[theorem]{Final Remark}
\newtheorem{definition}[theorem]{Definition}
\newtheorem{definitions}[theorem]{Definitions}
\allowdisplaybreaks[4]

\newcommand{\cvf}{\overset{\omega}{\rightarrow}}

\newcommand {\R}{\mathbb{R}}

\newcommand {\N} {\mathbb{N}}
\newcommand{\norma}[1]{\| #1 \|}
\newcommand{\conj}[2]{\left \{ {#1} \, : \, {#2} \right \}}

\begin{document}

\title{Compact positive multilinear operators on Banach lattices}
\author{Geraldo Botelho\thanks{Supported by Fapemig grants RED-00133-21 and APQ-01853-23.} ~and Vinícius C. C. Miranda\thanks{Supported by FAPESP grant 2023/12916-1 and Fapemig grant APQ-01853-23.\newline 2020 Mathematics Subject Classification: 46B42, 46B50, 46G25, 47B07, 47H60.
 }}
\date{}
\maketitle

\begin{abstract} Let $1 < p_1, \ldots, p_n < \infty, 1\leq q < \infty$ be such that $\sum\limits_{i=1}^n \frac{1}{p_i} < \frac{1}{q}$ and let $\mu_1, \ldots, \mu_n, \nu$ be arbitrary measures. Generalizing known linear and multilinear results, we prove that all positive $n$-linear operators from $\ell_{p_1} \times \cdots \times \ell_{p_n}$ to $L_q(\nu)$ and from $L_{p_1}(\mu_1) \times \cdots \times L_{p_1}(\mu_n)$ to $\ell_{q}$ are compact. This result, along with other related ones concerning free Banach lattices, shall emerge as consequences of some facts we prove about $M$-weakly compact multilinear operators on Banach lattices.
\end{abstract}

\section{Introduction}

A long tradition began in 1936 when Pitt \cite{pitt} proved that bounded linear operators from $\ell_p$ to $\ell_q$ are compact whenever $q < p$. Stepping  into the nonlinear environment, Pe\l{}czy\'nski \cite{pel} proved in 1957 that continuous $n$-homogeneous polynomials from $\ell_p$ to $\ell_q$ are compact if $nq < p$. After several related results, see, e.g. \cite{veronacho, raqueljesus, raqueljesus2}, in 1997 Alencar and Floret \cite{alencarfloret} proved the multilinear case: every continuous $n$-linear operator from $\ell_{p_1} \times \cdots \times \ell_{p_n}$ to $\ell_q$ is compact whenever $\sum\limits_{i=1}^n \frac{1}{p_i} < \frac{1}{q}$. Using Banach lattices techniques, Chen and Wickstead \cite{chenwick} proved in 1998 that positive linear operators from $\ell_p$ to $L_q(\nu)$ and from $L_p(\mu)$ to $\ell_q$ are compact if $q < p$. The main result of this paper, stated in the Abstract, can be regarded as a multilinear extension of this latter result and as a lattice counterpart of the former ones.

Next, we present 
two examples, one showing that the posivitity of the operator is essential, and the other one showing that atomicity is essential either in the domain spaces or in the target space.

\begin{examples}\rm \label{exintro} {\rm (1)} Using Holder's inequality and Khintchine’s inequalities \cite[Theorem 6.2.2]{albiac}, it is easy to see that $A((a_j)_j,(b_j)_j) = \sum\limits_{j=1}^\infty a_j b_j r_j$, where $(r_n)_n$ denotes the sequence of Rademacher functions, defines a non-compact continuous bilinear operator from $\ell_4 \times \ell_4$ to
$L_1([0,1])$. 

{\rm (2)} By \cite[Theorem 4.9]{chenwick}, there exists a positive non-compact operator $T: L_2([0,1]) \to L_1([0,1])$. Let $\varphi : \ell_3 \to \R$ be a positive linear functional functional. Then, $B(f, a) = \varphi(a) T(f)$ defines a non-compact positive bilinear operator from $L_2([0,1]) \times \ell_3$ to $L_1([0,1])$. 
\end{examples}

In our way to prove the main result we had to consider multilinear generalizations of the classical class of $M$-weakly compact linear operators. Actually, the main results of the paper are consequences of the results we prove for $M$-weakly compact multilinear operators.

In Section 2 we recall some basic facts about indices of Banach lattices and we prove some results that shall be needed later. The results about $M$-weakly compact multilinear operators are proved in Section 3. The main results, including the one stated in the Abstract and results concerning free Banach lattices, are proved in Section 4.

By $E^+$ we denote the positive cone of the Banach lattice $E$ and by $B_X$ the closed unit ball of the Banach space $X$. Throughout the paper, all measures are positive.

Given Banach spaces $X_1, \dots, X_n$ and $Y$, the Banach space of all continuous $n$-linear operators $A: X_1 \times \cdots \times X_n \to Y$ is denoted by $\mathcal{L}(X_1, \dots, X_n; Y)$. An operator $A \in \mathcal{L}(X_1, \dots, X_n; Y)$ is compact if $A(B_{E_1}\times \cdots \times B_{E_n})$ is a relatively compact subset of $Y$. Compact multilinear operators started being studied in   Pe\l{}czy\'nski \cite{pel}. 
For given Banach lattices $E_1, \dots, E_n$ and $F$, an $n$-linear operator $A: E_1 \times \cdots \times E_n \to F$ is said to be positive if $A(x_1, \dots, x_n) \geq 0$ for all $x_1 \in E_1^+, \dots, x_n \in E_n^+$. It is a well-known fact that $|A(x_1, \dots, x_n)| \leq A(|x_1|, \dots, |x_n|)$ for a positive $n$-linear operator $A: E_1 \times \cdots \times E_n \to F$ and all $x_1 \in E_1, \dots, x_n \in E_n$. The difference of two positive $n$-linear operators is called a regular $n$-linear operator, and the set of all regular $n$-linear operators from $E_1 \times \cdots \times E_n $ into $F$ is denoted by $\mathcal{L}^r(E_1, \dots, E_n; F)$. It is well known that positive (hence regular) multilinear operators are automatically continuous. Whenever $F$ is Dedekind complete, $\mathcal{L}^r(E_1, \dots, E_n; F)$ is a Banach lattice with the regular norm $\norma{A}_r := \norma{|A|}$, where $|A|$ denotes the absolute value of the regular $n$-linear operator $A: E_1 \times \cdots \times E_n \to F$. 

For (spaces of) continuous multilinear operators between Banach spaces we refer to \cite{dineen}, for (spaces of) regular multilinear operators between Banach lattices we refer to \cite{bubuskes, loane}, and for the basic theory of Banach lattices we refer to \cite{alip, meyer}.

\section{Preliminary results}

The following terminology was introduced by P. G. Dodds \cite{dodds77}:


\begin{definitions}\rm
    Let $E$ be a Banach lattice and let $1 \leq p \leq \infty$ be given. \\
    {\rm (1)} $E$ is said to have the {\it $\ell_p$-composition property} if every positive disjoint sequence $(x_n)_n$ in $E$ such that $(\norma{x_n})_n \in \ell_p$ satisfies $\displaystyle \sup_{n \in \N}\norma{x_1 + \cdots + x_n} < \infty$. The {\it lower index} $s(E)$ of $E$ is defined by
    $$ s(E) = \sup \conj{p \geq 1}{E \text{ has the $\ell_p$-composition property}}. $$
    {\rm (2)} $E$ is said to have the {\it $\ell_p$-decomposition property} if every positive disjoint order bounded sequence $(x_n)_n$ in $E$ satisfies $(\norma{x_n})_n \in \ell_p$.
    The {\it upper index} $\sigma(E)$ of E is defined by
    $$ \sigma(E) = \inf \conj{p \geq 1}{E \text{ has the $\ell_p$-decomposition property}}. $$
\end{definitions}

Next, we recall some properties related to the notions defined above.

\begin{remarks}\rm  \label{rema}
    {\rm (1)} Every Banach lattice has the $\ell_1$-composition property \cite[p. 74]{dodds77} and the $\ell_\infty$-decomposition property \cite[p. 75]{dodds77}. \\
    {\rm (2)} If $E$ has the $\ell_p$-composition property for some $p > 1$, then $E$ also has the $\ell_r$-composition property for every $1 \leq r \leq p$ \cite[p. 74]{dodds77}. On the other hand, if $E$ has the $\ell_p$-decomposition property for some $1 \leq p < \infty$, then $E$ also has the $\ell_r$-decomposition property for every $\infty \geq r \geq p$ \cite[p. 75]{dodds77}. These observations show that the lower and the upper indices are well defined. \\
    {\rm (3)} If $E$ has the $\ell_p$-composition property for some $p > 1$, then $E^*$ has order continuous norm \cite[Theorem 2.3]{dodds77}. \\
    {\rm (4)} If $E$ has the $\ell_p$-decomposition property for some $1 \leq p < \infty$ and $E$ has the $\ell_r$-composition property for some $1 < r \leq \infty$, then $E$ is reflexive \cite[Corollary 2.6]{dodds77}. \\
    {\rm (5)} Suppose that $\frac{1}{p} + \frac{1}{q} =  1$. Then, $E$ has the $\ell_p$-decomposition property if and only if $E^*$ has the $\ell_q$-composition property \cite[Theorem 2.14]{dodds77}. Also, $E$ has the $\ell_p$-composition property if and only if $E^*$ has the $\ell_q$-decomposition property \cite[p. 315]{doddsfremlin}.  \\
    {\rm (6)} It follows from (5) and (3) that if $E$ has the $\ell_p$-decomposition property, then $E^{**}$ has order continuous norm. Since $E$ is a closed sublattice of $E^{**}$, we get that $E$ also has order continuous norm. \\
    {\rm (7)} It follows from \cite[p. 314]{doddsfremlin} that for a Banach lattice $E$, the following are equivalent:

    {\rm (i)} $E$ has the $\ell_p$-composition property ($1 < p < \infty$).

    {\rm (ii)} For every disjoint norm-bounded sequence $(x_n)_n$ in $ E$ there exists a bounded linear operator $S:\ell_p \to E$ such that $S(e_n) = x_n$ for every $n \in \N$.

    {\rm (iii)} There exists a constant $M > 0$ such that $\left\|\sum\limits_{i=1}^n x_i\right\| \leq M \norma{(x_1, \dots, x_n)}_p$ holds for every finite disjoint subset $\{x_1, \dots, x_n \}$ of $E$.

    \noindent Item (iii) above coincides with the definition of the so-called {\it strong $\ell_p$-composition property} from \cite[Definition 2.7]{dodds77} and \cite[Definition 1.2]{schep}. See also \cite[Definition 1.f.4]{lind}.\\
   {\rm (8)} For every Banach lattice $E$, $E$ has the $\ell_p$-decomposition property if and only if $E$ has the so-called {\it strong $\ell_p$-decomposition property}, meaning that there exists a constant $M > 0$ such that $\norma{(x_1, \dots, x_n)}_p \leq M \left\|\sum\limits_{i=1}^n x_i\right\|$ holds for every finite disjoint subset $\{x_1, \dots, x_n \}$ of $E$ \cite[p. 78]{dodds77}. See also \cite[Definition 1.1]{schep} and \cite[Definition 1.f.4]{lind}.   \\
   {\rm (9)} For every Banach lattice $E$, it holds that $1 \leq s(E) \leq \sigma(E) \leq \infty$; this justifies why $s(E)$ and $\sigma(E)$ are called, respectively, the lower index and the upper index of $E$. Moreover, $s(E) = 1$ and $\sigma(E) = \infty$ for every finite-dimensional $E$ \cite[Theorem 3.2]{dodds77}. \\
   {\rm (10)} Items (7) and (8) above yield that, for every $E$,
   $$ s(E) = \sup \conj{p \geq 1}{E \text{ has the strong $\ell_p$-composition property}} $$
   and
   $$ \sigma(E) = \inf \conj{p \geq 1}{E \text{ has the strong $\ell_p$-decomposition property}}. $$
    {\rm (11)} It follows from \cite[p. 314]{doddsfremlin} that $\displaystyle \frac{1}{\sigma(E)} + \frac{1}{s(E^*)} = \frac{1}{\sigma(E^*)} + \frac{1}{s(E)} = 1$ holds for every Banach lattice $E$.
\end{remarks}

We recall that a Banach lattice $E$ is said to be an {\it abstract $L_p$-space} ($1 \leq p < \infty$) if $\norma{x+y}^p = \norma{x}^p + \norma{y}^p$ holds for all positive disjoint elements $x, y \in E$. In particular, the norm of $E$ is $p$-additive.
It is well known that $L_p(\mu) :=L_p(\Omega, \Sigma,\mu)$ is an abstract $L_p$-space for every measure space $(\Omega, \Sigma,\mu)$. Conversely, if $E$ is an abstract $L_p$-space, then there exists a topological Hausdorff space $X$ and a Baire measure $\mu$ on $X$ such that $E$ is isometrically isomorphic to $L_p(\mu)$ \cite[2.7.1]{meyer}. More details can be found in  \cite[Section 2.7]{meyer}. We also recall that a Banach lattice $E$ is an \textit{abstract $M$-space} ($AM$-space, in short) if $\norma{x \vee y} = \max\{\norma{x}, \norma{y}\}$ for all $x,y \geq 0$ in $E$. In this case the norm of $E$ is called an abstract $M$-norm.
The following characterizations were proven in \cite[Section 4]{dodds77}:

\begin{proposition} \label{teolp}
    {\rm (1)} A Banach lattice $E$ has the $\ell_1$-decomposition property if and only if has an equivalent abstract $L_1$-space norm. \\
    {\rm (2)} Let $1 \leq p < \infty$. A Banach lattice $E$ has the $\ell_p$-decomposition property and the $\ell_p$-composition property if and only if $E$ has an equivalent abstract $L_p$-space norm. \\
    {\rm (3)} Let $E$ be a $\sigma$-Dedekind complete Banach lattice (or equivalently, $E$ has the so-called principal projection property \cite[p. 18]{meyer}). Then, $E$ has the $\ell_\infty$-composition property if and only if $E$ has an equivalent abstract $M$-space norm.
\end{proposition}

As a consequence of Proposition \ref{teolp}, we have $s(L_p(\mu)) = \sigma(L_p(\mu)) = p$ for every $1 \leq p < \infty$ and every measure $\mu$, and  $\sigma(E) = s(E) = \infty$ for every $\sigma$-Dedekind complete $AM$-space $E$. Let us see one more example:

\begin{example} \label{exindices}\rm  Let $F$ be a Lorentz sequence space with a lower $2$-estimate, or equivalently, with the $\ell_2$-decomposition property, that is not $2$-concave (see, e.g., \cite[Example 1.f.19]{lind}). Also, recall that $F$ is a Banach lattice with the order induced by
its $1$-unconditional basis. We claim that $\sigma(F) = 2$. Indeed, since $F$ has the $\ell_2$-decomposition property, $\sigma(F) \leq 2$. If $F$ has the $\ell_r$-decomposition property for some $r < 2$, it follows from \cite[Theorem 1.f.7]{lind} that $F$ is $q$-concave for every $q \in (r, \infty)$, which is a contradiction because $F$ fails to be $2$-concave. Thus,
$\sigma(F) = 2$.
\end{example}

The following notion was introduced by Pełczy\'nski \cite{pel}.
For $0 \leq \alpha < 1$, a sequence $(x_n)_n$ in a Banach space is said to be $\tau_\alpha$-convergent to $0$ if there exists $c> 0$ such that
$$ \left\|\sum_{n \in B} x_n \right\| \leq c |B|^\alpha $$
for every finite subset $B \subset \N$, where $|B|$ denotes the cardinality of $B$. The application of
\cite[Main Theorem]{alencarfloret}  we prove next shall be needed later.


\begin{proposition}\label{lema2}
    Let $E_1, \dots, E_n$ and $F$ be Banach lattices and let $A \in \mathcal{L}(E_1, \dots, E_n; F)$ be given. Suppose that, for each $i=1, \dots, n$, $E_i$ has the $\ell_{p_i}$-composition property ($1 < p_i < \infty$). Then, for all norm bounded disjoint sequences $(x_1^k)_k$  in $E_1, \dots, (x_n^k)_k$ in $ E_k$, the sequence $(A(x_1^k, \dots, x_n^k))_k$ is $\tau_{\alpha}$-convergent to $0$ in $F$, where $\alpha = \sum\limits_{i=1}^n \frac{1}{p_i}$. In particular, $(A(x_1^k, \dots, x_n^k))_k$ is a weakly null sequence in $F$.
\end{proposition}

\begin{proof}
    Let  $(x_1^k)_k \subset E_1, \dots, (x_n^k)_k \subset E_k$ be norm bounded disjoint sequences. For each $i = 1, \dots, n$, since $E_i$ has the $\ell_{p_i}$-composition property, there exists a bounded linear operator $S_i: \ell_{p_i} \to E_i$ such that $S_i(e_k) = x_i^k$ for every $k \in \N$ [Remark \ref{rema}(7)]. Hence,
    $$ \left\|\sum_{k=1}^\infty a_i^k x_i^k\right\| = \left\|\sum_{k=1}^\infty a_i^kS_i(e_k)\right\| = \left\|S_i\left(\sum_{k=1}^\infty a_i^k e_i^k\right)\right\| \leq \norma{S_i} \!\cdot\!\norma{(a_i^k)_k}_{p_i}  $$
    holds for every $(a_i^k)_k \in \ell_{p_i}$ and each $i = 1, \dots, n$. This proves that each $(x_i^k)_k$ has the so-called upper $p$-estimate  \cite[Proposition 2.2]{alencarfloret}, therefore $(x_i^k)_k$ is $\tau_{1/p_i}$-convergent to $0$ for every $i = 1, \dots, n$ \cite[2.2]{alencarfloret}. By \cite[Main Theorem]{alencarfloret}, we conclude that $(A(x_1^k, \dots, x_n^k))_k$ is $\tau_\alpha$-convergent to $0$ in $F$ for where $\alpha = \sum\limits_{i=1}^n \frac{1}{p_i}$. For the second statement, see \cite[2.2]{alencarfloret}.
\end{proof}

We conclude this section with one more result that will be needed in the next section.

\begin{lemma} \label{lema3}
    Let $E_1, \dots, E_n$ and $F$ be Banach lattices such that $ \sum\limits_{i=1}^n \frac{1}{s(E_i)} < 1$. Then, there exist $r_1 \leq s(E_1), \dots, r_n \leq s(E_n)$ such that each $E_i$ has the $\ell_{r_i}$-composition property and $ \sum\limits_{i=1}^n \frac{1}{r_i} < 1$.
\end{lemma}

\begin{proof}
    Letting $K = \sum\limits_{i=2}^n \frac{1}{s(E_i)}$, we have $\frac{1}{s(E_1)} + K < 1$, so $\frac{1}{1 - K} < s(E_1)$. By the definition of $s(E_1)$, there exists a real number $r_1$ with $\frac{1}{1 - K} < r_1 \leq s(E_1)$ such that $E_1$ has the $\ell_{r_1}$-composition property, which yields that
    $  \frac{1}{r_1} + \sum\limits_{i=2}^n \frac{1}{s(E_i)} < 1$. Just repeat the argument for $i = 2, \dots, n$, to obtain the result.
\end{proof}

\section{$M$-weakly compact multilinear operators}

Recall that a bounded linear operator 
from a Banach lattice to a Banach space is said to be {\it $M$-weakly compact} if it maps norm bounded disjoint sequences to norm null sequences (see, e.g., \cite[Definition 3.6.9(iv)]{meyer}). It is a natural line of investigation in Functional Analysis to study multilinear versions of already studied classes of linear operators. Reinforcing that this line of investigation may be fruitful, the main results of this paper will be derived in the next section from results about two types of  $M$-weakly compact multilinear operators we introduce in this section. The first one is the following:

\begin{definition} \label{defmwc}\rm
Let $E_1, \dots, E_n$ be Banach lattices, and let $X$ be Banach space. An $n$-linear operator $A: E_1 \times \dots \times E_n \to X$ is said to be {\it $M$-weakly compact} if $\norma{A(x_1^k, \dots, x_n^k)} \to 0$ for all disjoint sequences $(x_1^k)_k$ in $ B_{E_1}, \dots, (x_n^k)_k$ in $ B_{E_k}$.
\end{definition}

According to the standard terminology, $A: E_1 \times \dots \times E_n \to X$ is said to be {\it separately $M$-weakly compact} if, for every $i \in \{1, \ldots,n\}$ and all $x_j \in E_j$, $j \neq i$, the linear operator $x_i \in E_i \mapsto A(x_1, \ldots, x_n) \in F $ is M-weakly compact.

\begin{example}\label{9ol2}\rm From \cite[Theorem 2.4.14]{meyer} we know that the dual $E^*$ of a Banach lattice $E$ has order continuous norm if and only if every linear functional on $E$ is $M$-weakly compact.

Let $E$ and $F$ be Banach lattices such that $E^*$ has order continuous norm, fix functionals $x^* \in E^*$ and $y^* \in F^*$, and consider the continuous bilinear form
$$B: E \times F \to \R~,~ B(x,y) = x^*(x)y^*(y).$$
On the one hand, $B$ is $M$-weakly compact  for any choice of $x^*$ and $y^*$. On the other hand, if $F^*$ fails to have order continuous norm, then we can choose $0 \neq x^* \in E^*$ and a non $M$-weakly compact linear functional $y^* \in F^*$. 
{In this case, fixing $x_0 \in E$ such that $x^*(x_0) \neq 0$, the resulting linear linear functional $y \in F \mapsto B(x_0, y)$ fails to be $M$-weakly compact.} In particular, $B$ is not separately $M$-weakly compact.
\end{example}

The example above suggests our second generalization of $M$-weakly compact linear operators:

\begin{definition} \label{defstrong}\rm
Let $E_1, \dots, E_n$ be Banach lattices and let $X$ be Banach space. An $n$-linear operator $A: E_1 \times \dots \times E_n \to X$ is said to be  {\it strongly $M$-weakly compact} if, fixing any $k \in \{0, \dots, n-1\}$ variables, the resulting $(n-k)$-linear operator is $M$-weakly compact.
\end{definition}

In the definition above, the case $k = 0$ means that $A$ is $M$-weakly compact, whereas the case $k = n-1$ means that $A$ is separately $M$-weakly compact. In particular: (i) Every strongly $M$-weakly compact operator is $M$-weakly compact and separately $M$-weakly compact. (ii) A bilinear operator is $M$-weakly compact if and only if it is $M$-weakly compact and separately $M$-weakly compact. For $n \geq 3$, this equivalence is no longer true in general: for a trilinear operator $A: E_1 \times E_2 \times E_3 \to F$ to be strongly $M$-weakly compact, for any fixed $x_1 \in E_1$, the bilinear operator
$$(x_2, x_3) \in E_2 \times E_3 \mapsto A(x_1, x_2, x_3) \in F, $$
must be $M$-weakly compact, a condition that does not follow automatically if $A$ is $M$-weakly compact and separately $M$-weakly compact.



As we saw in Example \ref{9ol2}, there are $M$-weakly compact bilinear forms that fail to be strongly $M$-weakly compact. But, if $E_1^*, \dots, E_n^*$ have order continuous norms, then  for all $x_1^* \in E_1^*, \dots, x_n^* \in E_n^*$, the $n$-linear form
$$ A: E_1 \times \cdots \times E_n \to \R~,~ A(x_1, \dots, x_n) = x_1^*(x_1)\cdots x_n^*(x_n), $$
is strongly $M$-weakly compact. Anyway, Definition \ref{defstrong} seems to be too demanding, that is, the class of strongly $M$-weakly compact multilinear operators seems to be very small. Nevertheless, we shall provide soon examples of Banach lattices $E_1, \dots, E_n$ and $F$ for which every (regular) $n$-linear operator from $E_1 \times \cdots \times E_n$ to $F$ is strongly $M$-weakly compact. Our first result in this direction is a multilinear version of \cite[Proposition 4.1]{chenwick}.

\begin{theorem} \label{teo1}
    Let $E_1, \dots, E_n$ be Banach lattices with $2n < s(E_i) < \infty$ for every $i =1, \dots, n$, and let $F$ be an abstract $L_q$-space with $1 \leq q \leq 2$. If $F$ does not contain any atoms, then every $A \in \mathcal{L}(E_1, \dots, E_n; F)$ is strongly $M$-weakly compact.
\end{theorem}

\begin{proof}
    Notice first that it suffices us to check that each $A \in \mathcal{L}(E_1, \dots, E_n; F)$ is $M$-weakly compact in the sense of Definition \ref{defmwc}. Indeed, fixing any $k \in \{1, \dots, n-1\}$ variables, the remaining $(n-k)$-linear operator is defined on $E_{i_1} \times \dots \times E_{i_{n-k}}$ for some $i_1, \ldots, i_{n-k} \in \{1, \ldots, n\}$, and $ \infty > s(E_{i_j}) > 2n > 2(n-k)$ for every $j = 1, \dots, n-k$. As mentioned before, we may assume that $F = L_q(\mu)$ for some measure $\mu$. Suppose, for the sake of contradiction, that there exists a non-$M$-weakly compact operator $A \in \mathcal{L}(E_1, \dots, E_n; F)$. 
    In this case there exist disjoint sequences
     $(x_1^k)_k \subset B_{E_1}, \dots, (x_n^k)_k \subset B_{E_k}$ such that $\displaystyle \lim_{k\to \infty} \norma{A(x_1^k, \dots, x_n^k)} \neq 0$. Thus, there exist $\varepsilon > 0$ and a subsequence $(k_j)_j$ of $\N$ such that $\norma{A(x_1^{k_j}, \dots, x_n^{k_j})} \geq \varepsilon$ for every $j \in \N$. Choosing $ \displaystyle 2n < r < \min_{1 \leq i \leq n} s(E_i)$, we get that $E_i$ has the property $\ell_r$-composition property for all $i =1, \dots, n$. Hence, for each $i =1, \dots, n$, there exists an operator $S_i: \ell_r \to E_i$ such that $S_i(e_j) = x_i^{k_j}$ for every $j \in \N$. In particular,
     $$ \norma{A(S_1(e_{j}), \dots, S_n(e_{j}))} = \norma{A(x_1^{k_j}, \dots, x_n^{k_j})} \geq \varepsilon. $$
     Taking $1 < p < \frac{r}{2n}$, it is easy to see that the series $\sum\limits_{j=1}^\infty j^{-p/r} e_{k_j}$ is unconditionally convergent in $\ell_r$ and that $\sum\limits_{j=1}^\infty j^{-2np/r} = + \infty$. From \cite[Proposition 8.3]{defantfloret} it follows that
     $$(j^{-p/r} e_{k_j})_j \in \ell_1^u(\ell_r) := \conj{(z_j)_j \in \ell_1^w(\ell_r)}{ \sup_{\varphi \in B_{\ell_r^*}} \sum_{j=n}^\infty|\varphi(z_j)| \overset{n\to\infty}{\longrightarrow} 0},$$
     where $\ell_1^w(\ell_r)$ denotes the collection of weakly absolutely summable sequences in $\ell_r$. Beware that \cite{defantfloret} uses the symbol $\ell_1^{w,0}$ instead of $\ell_1^u$. Therefore, for each $i =1, \dots, n$, $(S_i(j^{-p/r} e_{j}))_j \in \ell_1^u(F)$, and by \cite[Theorem 4.3]{botelhocampos} we get that
     $$(A(S_1(j^{-p/r} e_{j}), \dots, S_n( j^{-p/r}e_{j})))_j \in \ell_1^u(F).$$ Calling on \cite[Proposition 8.3]{defantfloret} once again, we have that the series
     $$\displaystyle \sum_{j=1}^\infty j^{-np/r} A(S_1(e_{j}), \dots, S_n(e_{j})) = \sum_{j=1}^\infty A(S_1(j^{-p/r} e_{k_j}), \dots, S_n( j^{-p/r}e_{k_j}))$$ is unconditionally convergent in $F$. 
     For each $j \in \N$, set
     $$y_j := j^{-np/r} A(S_1(e_{k_j}), \dots, S_n(e_{k_j})) \in F. $$ Putting $X := \overline{ \text{span}} \{y_j: j \in \N\}$, it follows from \cite[Exercise 9, p.\,204]{alip} that there exists a separable Banach sublattice $G$ of $F$ containing $X$. Moreover, since $F$ is an abstract $L_q$-space with no atoms, $G$ is a separable Banach lattice without atoms and with a $q$-additive norm. By \cite[Theorem 2.7.3]{meyer}, $G$ is isometrically lattice isomorphic to $L_q[0,1]$. Thus, $(y_j)_j$ is unconditionally summable, and by a comment at the bottom of page 23 in \cite{diestel}, we get that
    $$\infty > \sum_{j=1}^\infty \norma{y_j}^2 = \displaystyle \sum_{j=1}^\infty \norma{j^{-np/r} A(S_1(e_{k_j}), \dots, S_n(e_{k_j}))}^2 \geq \varepsilon^2 \sum_{j=1}^\infty j^{-2np/r} = \infty, $$
    a contradiction that completes the proof.
\end{proof}

Let us see that Theorem \ref{teo1} is false if we either drop the assumptions on the lower indices of $E_i$ or if we take $q > 2$.

\begin{examples}\rm
{\rm (1)} The bilinear operator $A: \ell_4 \times \ell_4 \to L_1([0,1])$ from Example \ref{exintro}(1) is not $M$-weakly compact because $\norma{A(e_k, e_k)}_1 = \norma{r_k}_1 = 1$ holds for every $k \in \N$. Thus, Theorem \ref{teo1} is false if we take that $s(E_i) = 2n$.\\
{\rm (2)}  For each $q > 2$, there exists an embedding $T: \ell_q \to L_q([0,1])$ (see \cite[Proposition 6.4.3]{albiac}). So, $A((a_j)_j, (b_j)_j) = \sum\limits_{j=1}^\infty a_j b_j T(e_j)$ defines a continuous bilinear operator from $\ell_q \times \ell_\infty$ to $L_q([0,1])$. As $(e_k)_k$ is not norm null and $T$ is an isomorphism onto its range, $(Te_k)_k$ is not norm null in $L_q([0,1])$. Since $A(e_k, e_k) = T(e_k)$ for every $k \in \N$, we conclude that $A$ fails to be $M$-weakly compact. Thus, Theorem \ref{teo1} is false if we take $q > 2$.
\end{examples}

Our next purpose is to prove the following:

\begin{theorem} \label{teo2}
Let $E_1, \dots, E_n$ and $F$ be Banach lattices such that $\sum\limits_{i=1}^n \frac{1}{s(E_i)} < \frac{1}{\sigma(F)}$. \\
{\rm (1)}  If $F$ is atomic, then every $A \in \mathcal{L}(E_1, \dots, E_n; F)$ is strongly $M$-weakly compact. \\
{\rm (2)} Every $A \in \mathcal{L}^r(E_1, \dots, E_n; F)$ is strongly $M$-weakly compact.
\end{theorem}

A multilinear version of \cite[Theorem 7.2]{doddsfremlin} is needed to prove Theorem \ref{teo2}. To do so, we will need the following two lemmas. The first one is a multilinear version of an argument used in the proof of \cite[Main Theorem]{ando}.

\begin{lemma} \label{lemaaux}
    Let $X_1, \dots, X_n$ and $Z$ be Banach spaces, let $A \in \mathcal{L}(X_1, \dots, X_n; Z)$ be given, let $(x_1^k)_k \subset X_1, \dots, (x_n^k)\subset X_n$ be weakly null sequences, and let $(f_k)_k$ be a weak$^*$ null sequence in $Z^*$. Then, there exists a subsequence $(k_j)_j$ of $\N$ such that $|f_k(A(x_1^{k_{j_1}}, \dots, x_n^{k_{j_n}}))| \leq 2^{- \max \{l, j_1, \dots, j_n\}}$ whenever $(l, j_1, \dots, j_n)$ has at least two different coordinates.
\end{lemma}

\begin{proof}
    We prove the  case $n = 2$. Let $X, Y$ and $Z$ be Banach spaces, let $A \in \mathcal{L}(X, Y; Z)$, let $(x_k)_k \subset X$ and $(y_k)_k \subset Y$ be weakly null sequences, and let $(f_k)_k$  be a weak$^*$ null sequence in $Z^*$. Choose $k_1 = 1$.
    Since
    $$ |f_k(A(x_{k_1}, y_{k_1}))| + |f_{k_1}(A(x_k, y_{k_1}))| + |f_{k_1}(A(x_{k_1}, y_k))| \overset{k \to \infty}{\longrightarrow} 0, $$
    there exists $k_2 > k_1$ such that
    $$ |f_{k_2}(A(x_{k_1}, y_{k_1}))| + |f_{k_1}(A(x_{k_2}, y_{k_1}))| + |f_{k_1}(A(x_{k_1}, y_{k_2}))| \leq 2^{-2}. $$
    Now, from
    $$ \sum_{i,j = 1}^2 |f_k(A(x_{k_i}, y_{k_j}))| + \sum_{i,j = 1}^2 |f_{k_i}(A(x_k, y_{k_j}))| + \sum_{i,j = 1}^2 |f_{k_i}(A(x_{k_j}, y_k))| \overset{k \to \infty}{\longrightarrow} 0, $$
    there is $k_3 > k_2$ such that
    $$ \sum_{i,j = 1}^2 |f_{k_3}(A(x_{k_i}, y_{k_j}))| + \sum_{i,j = 1}^2 |f_{k_i}(A(x_{k_3}, y_{k_j}))| + \sum_{i,j = 1}^2 |f_{k_i}(A(x_{k_j}, y_{k_3}))| \leq 2^{-3}. $$
    So far, we have $k_3 > k_2 > k_1$ such that $\displaystyle |f_{k_l} (x_{k_i}, y_{k_j}) | \leq 2^{-\max \{l, i, j \}}$ whenever $(l,i,j) \in \{1,2,3\}^3$ has at least two different coordinates.
    Suppose that $k_1 < k_2 < \cdots < k_N$ have been chosen such that $ |f_{k_l} (x_{k_i}, y_{k_j}) | \leq 2^{-\max \{l, i, j \}}$  whenever $(l,i,j) \in \{1, \dots, N\}^3$ has at least two different coordinates. From the convergence
    $$ \sum_{i,j = 1}^N |f_k(A(x_{k_i}, y_{k_j}))| + \sum_{i,j = 1}^N |f_{k_i}(A(x_k, y_{k_j}))| + \sum_{i,j = 1}^N |f_{k_i}(A(x_{k_j}, y_k))| \overset{k \to \infty}{\longrightarrow} 0, $$
    there exists $k_{N+1} > k_N$ such that
    $$ \sum_{i,j = 1}^N |f_{k_{N+1}}(A(x_{k_i}, y_{k_j}))| + \sum_{i,j = 1}^N |f_{k_i}(A(x_{k_{N+1}}, y_{k_j}))| + \sum_{i,j = 1}^N |f_{k_i}(A(x_{k_j}, y_{k_{N+1}}))| \leq 2^{-(N+1)}. $$
    Combining the above inequality with the induction hypothesis, we get  $\displaystyle |f_{k_l} (x_{k_i}, y_{k_j}) | \leq 2^{-\max \{l, i, j \}}$ whenever $(l,i,j) \in \{1, \dots, N+1\}^3$ has at least two different coordinates, and we are done.
\end{proof}


\begin{lemma} \label{lemaholder}
 Let $n \geq 2$ be an integer.    If $r_1, \dots, r_n > 1$ are such that $\frac{1}{r_1} + \cdots + \frac{1}{r_n} < 1$, then there are $(a_1^j)_j \in \ell_{r_1}^+, \dots, (a_n^j)_j \in \ell_{r_n}^+$ such that $\sum\limits_{j=1}^\infty a_1^n \cdots a_1^n = + \infty$.
\end{lemma}

\begin{proof}
    Letting $p =  \dfrac{1}{\frac{1}{r_1} + \cdots + \frac{1}{r_{n-1}}}$, we have $\frac{1}{p} + \frac{1}{r_n} < 1$. Hence, $r_n > p^*$,  where $p^*$ is the conjugate exponent of $p$, which yields that $\ell_{r_n} \not\subset \ell_{p^*}$. Take $(b_j)_j \in \ell_{q} \setminus \ell_{p^*}$. By a classical application of Banach's Steinhauss Theorem (Principle of Uniform Boundedness), there exists $(x_j)_j \in \ell_p$ such that $ \sum\limits_{j=1}^\infty x_j b_j = +\infty$. For each $i = 1, \dots, n-1$, define $a_j^i = x_j^{p/r_i}$ for every $j \in \N$. It is easy to see that $x_j = a_1^j \cdots a_{n-1}^j$ for every $j \in \N$ and $(a_i^j)_j \in \ell_{r_i}^+$ for each $i = 1, \dots, n-1$.  Therefore, $\sum\limits_{j=1}^\infty a_1^j \cdots a_{n-1}^jb_j = + \infty$.
\end{proof}

Next we prove a multilinear version of \cite[Theorem 7.2]{doddsfremlin}.

\begin{proposition} \label{lemamweak2}
    Let $E_1, \dots, E_n$ and $F$ be Banach lattices such that $\sum\limits_{i=1}^n \frac{1}{s(E_i)} < \frac{1}{\sigma(F)}$. Then $\displaystyle \lim_{k \to \infty} y_k^\ast(A(x_1^k, \dots, x_n^k)) = 0$ for every $A \in \mathcal{L}(E_1, \dots, E_n; F)$ and all positive disjoint norm bounded sequences $(x_1^k)_k$ in $E_1, \dots, (x_n^k)$ in $E_n$, $(y_k^\ast)$ in $ F$.
\end{proposition}

\begin{proof}
    We notice first that, since $\sum\limits_{i=1}^n \frac{1}{s(E_i)} < \frac{1}{\sigma(F)}$, we have $\sigma(F) < \infty$ and $s(E_i) > 1$ for all $i = 1, \dots, n$. This implies that $E_1^\ast, \dots, E_n^\ast$ and $F$ have order continuous norm by Remark \ref{rema}.  Moreover, it follows from Remark \ref{rema}(11), that
    $$ \sum_{i=1}^n \frac{1}{s(E_i)} + \frac{1}{s(F^*)} = \sum_{i=1}^n \frac{1}{s(E_i)} + 1 - \frac{1}{\sigma(F)} < 1.  $$
    By Lemma \ref{lema3}, there are $r_1 \leq s(E_1), \dots, r_n \leq s(E_n)$ and $s \leq s(F^*)$ such that each $E_i$ has the $\ell_{r_i}$-composition property, $F$ has the $\ell_s$-composition property, and
    $ \sum\limits_{i=1}^n \frac{1}{r_i} + \frac{1}{s}< 1$. By Lemma \ref{lemaholder} there are positive sequences $(a_1^j)_j \in \ell_{r_1}, \dots, (a_n^j)_j \in \ell_{r_n}$ and $(b_j)_j \in \ell_s$ so that $\sum\limits_{j=1}^\infty  b_j a_{1}^{j} \cdots a_{n}^{j} = + \infty$.

    Suppose, for the sake of contradiction, that there are $A \in \mathcal{L}(E_1, \dots, E_n; F)$ and normalized disjoint sequences $(x_1^k)_k$ in $ E_1, \dots, (x_n^k)$ in $ E_n, (y_k^\ast)$ in $F^*$ so that $\displaystyle \lim_{k \to \infty} y_k^\ast(A(x_1^k, \dots, x_n^k)) \neq 0$. By passing to a subsequence if necessary, we may assume that there exists $\varepsilon > 0$ such that $|y_k^\ast(A(x_1^k, \dots, x_n^k))| \geq \varepsilon$ for every $k \in \N$. By replacing $x_1^k$ with $-x_1^k$ if necessary, we may assume that $y_k^\ast(A(x_1^k, \dots, x_n^k)) \geq \varepsilon$ holds for every $k \in \N$. Since $E_1^*, \dots, E_n^*$ and $F$ have order continuous norms, we get from \cite[Theorem 2.4.14 and Corollary 2.4.3]{meyer} that $(x_1^k)_k, \dots, (x_n^k)_k$ are weakly null in $E_1, \dots, E_n$, respectively, and $(y_k^*)_k$ is weak* null in $F^*$. Thus, by Lemma \ref{lemaaux}, we may assume, by passing to a subsequence if necessary, that $ |y_k^*(A(x_1^{j_1}, \dots, x_n^{j_n}))| \leq 2^{-\max\{ k, j_1, \dots, j_n \}} $ whenever $(k, j_1, \dots, j_n)$ has at least two different coordinates.
    For each $i =1, \dots, n$, $x_i : = \sum\limits_{j=1}^\infty a_i^j x_i^j$ converges in $E_i$ because $E_i$ has the $\ell_{r_i}$-composition property (the convergence of the series follows from Remark \ref{rema}(7)). Hence, letting $\displaystyle b = \sup_{j \in \N} b_j$, $a = \displaystyle \max_{1 \leq i \leq n} \sup_{j \in \N} a_i^j$ and $h_k = \displaystyle \sum_{j=1}^k b_j y_j^*$ for each $k \in \N$, we get
\begin{align*}
    |h_k(A(x_1, \dots, x_n))| & = \left |\sum_{j=1}^k b_j y_j^* (A(x_1, \dots, x_n)) \right | \\ &=  \left | \sum_{j=1}^k \sum_{j_1, \dots, j_n=1}^\infty  b_j a_{1}^{j_1} \cdots a_{n}^{j_n} \, y_j^*(A(x_1^{j_1}, \dots, x_n^{j_n})) \right | \\
    & \geq \left | \sum_{j=1}^k  b_j a_{1}^{j} \cdots a_{n}^{j} \, y_j^*(A(x_1^{j}, \dots, x_n^{j})) \right | - \\
    & - \sum_{j=1}^k \sum_{(j_1, \dots, j_n) \neq (j, \dots, j)} |b_j a_{1}^{j_1} \cdots a_{n}^{j_n} \, y_j^*(A(x_1^{j_1}, \dots, x_n^{j_n}))| \\
    & \geq \varepsilon \sum_{j=1}^k  b_j a_{1}^{j} \cdots a_{n}^{j} \, -  \sum_{(j,j_1, \dots, j_n) \neq (j, j, \dots, j)} b_j a_{1}^{j_1} \cdots a_{n}^{j_n} \, |y_j^*(A(x_1^{j_1}, \dots, x_n^{j_n}))| \\
    & \geq \varepsilon \sum_{j=1}^k  b_j a_{1}^{j} \cdots a_{n}^{j} \, -  b a^n \cdot \sum_{(j,j_1, \dots, j_n) \neq (j, j, \dots, j)} 2^{- \max \{ j, j_1, \dots, j_n \}} \\
    & \geq \varepsilon \sum_{j=1}^k  b_j a_{1}^{j} \cdots a_{n}^{j} \, -  b  a^n \cdot \sum_{l=1}^\infty 2^{-l} \to \infty  \text{ as $k \to \infty$}.
\end{align*}
However, since $F^*$ has the $\ell_s$-composition property, the limit $\displaystyle \lim_{k \to \infty} h_k = \sum_{j=1}^\infty b_j y_k^*$ exists in $F^*$ by Remark \ref{rema}(7). This contradiction completes the proof.
\end{proof}

Now we are in the position to prove Theorem \ref{teo2}.

\medskip

\noindent \textit{Proof of Theorem \ref{teo2}.}  By assumption, $E_1, \dots, E_n$ and $F$ be Banach lattices so that $\sum\limits_{i=1}^n \frac{1}{s(E_i)} < \frac{1}{\sigma(F)}$. We begin by noticing that it is enough to check that $A$ is $M$-weakly compact, because for every $k \in \{1, \dots, n-k\}$ and all indexes $i_1, \dots, i_{n-k}$, it holds
$$\displaystyle \sum_{j=1}^{n-k} \frac{1}{s(E_{i_j})} = \sum_{i=1}^n \frac{1}{s(E_i)} < \frac{1}{\sigma(F)}.$$
As in the proof of Proposition \ref{lemamweak2}, $F$ has order continuous norm, hence it is Dedekind complete, and, for each $i=1, \dots, n$, $E_i$ has the $\ell_{p_i}$-composition property ($1 < p_i < s(E_i)$). Let $A \in \mathcal{L}(E_1, \dots, E_n; F)$ be given and let $(x_1^k)_k \subset B_{E_1}, \dots, (x_n^k)_k \subset B_{E_k}$ be disjoint sequences. To prove that $\displaystyle \lim_{k \to \infty} A(x_1^k, \dots, x_n^k) = 0$, it is sufficient to prove that $|A(x_1^k, \dots, x_n^k)| \cvf 0$ in $F$ and that $\displaystyle \lim_{k \to \infty} y_k^\ast(A(x_1^k, \dots, x_n^k)) = 0$ for every positive norm bounded disjoint sequence $(y_k^*)_k \subset F^*$ (see \cite[Corollary 2.6]{doddsfremlin}). The second condition follows from Proposition \ref{lemamweak2}, leaving us to check that $|A(x_1^k, \dots, x_n^k)| \cvf 0$ in $F$. On the one hand,
Proposition \ref{lema2}  yields that $A(x_1^k, \dots, x_n^k) \cvf 0$ in $F$, and assuming that $F$ is atomic, we obtain from \cite[Proposition 2.5.23]{meyer} that $|A(x_1^k, \dots, x_n^k)| \cvf 0$ in $F$, proving statement (1) of the theorem.  On the other hand, supposing that $A$ is positive, we get from Proposition \ref{lema2} that $(A(|x_1^k|, \dots, |x_n^k|))_k$ is a weakly null sequence in $F$, and so the inequality
 $|A(x_1^k, \dots, x_n^k)| \leq A(|x_1^k|, \dots, |x_n^k|)$ yields that $|A(x_1^k, \dots, x_n^k)| \cvf 0$ in $F$ for any positive $n$-linear operator $A: E_1 \times \cdots \times E_n  \to F$, proving that every positive $n$-linear operator from $E_1 \times \cdots \times E_n$ into $F$ is $M$-weakly compact. Now, statement (2) of the theorem follows by decomposing a regular operator into its positive and  negative parts.
\qed

\bigskip

The following examples arise from Theorems \ref{teo1} and \ref{teo2}.

\begin{examples}\rm \label{exemplos1}
{\rm (1)} Let $n \in \N$ be given. By Theorem \ref{teo1}, every continuous $n$-linear operator $A: L_{p_1}(\mu_1) \times \cdots \times L_{p_n}(\mu_n) \to L_q([0,1])$ is $M$-weakly compact for all $p_1, \dots, p_n \in (2n, \infty)$, $1 \leq q \leq 2$, and all measures $\mu_1, \dots, \mu_n$.
\\
{\rm (2)} By Theorem \ref{teo2}(1), if $ \sum\limits_{i=1}^n \frac{1}{p_i} < \frac{1}{q}$, then every continuous $n$-linear operator $A: L_{p_1}(\mu_1) \times \cdots \times L_{p_n}(\mu_n) \to \ell_q(I)$ is strongly $M$-weakly compact for all measures $\mu_1, \dots, \mu_n$ and every index set $I$. \\
{\rm (3)} Fix $1 < p < 2$ and consider a Lorentz sequence $d(w,p)$ as a Banach lattice with the order induced by its $1$-unconditional basis. Then, $d(w,p)$ is atomic. It follows from
Theorem \ref{teo2}(1) and Example \ref{exindices} that, if $\sum\limits_{i=1}^n \frac{1}{p_i} < \frac{1}{2}$, then every continuous $n$-linear operator $A: L_{p_1}(\mu_1) \times \cdots \times L_{p_n}(\mu_n) \to d(w,p)$ is strongly $M$-weakly compact for all measures measures $\mu_1, \dots, \mu_n$. \\
{\rm (4)} If $\sum\limits_{i=1}^n \frac{1}{p_i} < \frac{1}{q}$, then every regular $n$-linear operator $A: L_{p_1}(\mu_1) \times \cdots \times L_{p_n}(\mu_n) \to L_q(\nu)$ is  strongly $M$-weakly compact for all measures $\mu_1, \dots, \mu_n, \nu$. \\
{\rm (5)} Given a Banach space $X$, the free Banach lattice generated by $X$ is a Banach lattice ${\rm FBL}[X]$ equipped with a linear isometric embedding $\phi_X: X \to {\rm FBL}[X]$ such that for every bounded linear operator from $X$ to an arbitrary Banach lattice $F$, there exists a lattice homomorphism $\widehat{T}: FBL[X] \to F$ such that $\widehat{T} \circ \phi_X = T$ and $\norma{\widehat{T}} = \norma{T}$. The notion of free Banach lattices appeared in \cite{aviles}. For recent developments, see \cite{garcialeung, garciapedro, oikhbergjfa, oikhberg}. Given $n \in \N$, $1 < p_1, \dots, p_n < 2$, $1 \leq q < \infty$, and measures $\mu_1, \dots, \mu_n, \nu$, we get from \cite[Corollary 9.31]{oikhberg} that
$$s({\rm FBL}[L_{p_i}(\mu_i)]) = p_i \text{ for every $i = 1, \dots, n$}.$$
By Theorem \ref{teo2},
every regular $n$-linear operator
$$ A: {\rm FBL}[L_{p_1}(\mu_1)] \times \cdots \times {\rm FBL}[L_{p_n}(\mu_n)] \to L_q(\nu) $$
is strongly $M$-weakly compact whenever $\sum\limits_{i=1}^n \frac{1}{p_i} < \frac{1}{q}$.
Also, every continuous $n$-linear operator
$$ A: {\rm FBL}[L_{p_1}(\mu_1)] \times \cdots \times {\rm FBL}[L_{p_n}(\mu_n)] \to \ell_q $$
is  strongly $M$-weakly compact whenever $\sum\limits_{i=1}^n \frac{1}{p_i} < \frac{1}{q}$. \\
{\rm (6)} Let $F$ be a Banach lattice with $2 \leq s(F) < \infty$, for instance $F = L_q(\nu)$ for every positive measure $\nu$ and every $q \geq 2$. By \cite[Corollary 9.31]{oikhberg}, $s({\rm FBL}[F]) = \min \{ 2, s(F) \} = 2$, and by Remark \ref{rema}(11) we obtain that $\sigma (({\rm FBL}[F])^*) = 2$. Thus, given $1 < p_1, \dots, p_n < \infty$ and measures $\mu_1, \dots, \mu_n$, we get from Theorem \ref{teo2} that every regular $n$-linear operator
$$ A: L_{p_1}(\mu_1) \times \cdots \times L_{p_n}(\mu_n) \to  ({\rm FBL}[F])^* $$
is strongly $M$-weakly compact whenever $\sum\limits_{i=1}^n \frac{1}{p_i} < \dfrac{1}{\sigma(({\rm FBL}[F])^*)} = \frac{1}{2}$.

\end{examples}

\section{Main results}

To prove our main result, the one stated in the Abstract, we shall use the following theorem that gives sufficient conditions for a positive strongly $M$-weakly compact multilinear operator to be compact. Throughout this section,  $E_1, \ldots, E_n$ and $F$ are Banach lattices.

\begin{theorem} \label{corcomp}
    Let $A: E_1 \times \cdots \times E_n \to F$ be a positive strongly $M$-weakly compact $n$-linear operator. If one of the following conditions hold, then $A$ is compact: \\
    {\rm (1)} $E_1, \dots, E_m$ are atomic with order continuous norms. \\
    {\rm (2)} $F$ is atomic with order continuous norm.
\end{theorem}

In order to prove Theorem \ref{corcomp}, we will need the following two lemmas.


\begin{lemma} \label{lemacomp1}
  If $A: E_1 \times \cdots \times E_n \to F$ is a positive $M$-weakly compact $n$-linear operator, then for all norm bounded sets $A_1 \subset E_1, \dots, A_n \subset E_n$ and $\varepsilon > 0$, there exist $y_1 \in E_1^+, \dots, y_n \in E_n^+$ so that  $$ \norma{A((|x_1| - y_1)^+, \dots, (|x_n| - y_n)^+)} < \varepsilon \quad \text{for all $x_1 \in A_1, \dots, x_n \in A_n$}.$$
\end{lemma}

\begin{proof}
    Assuming that the thesis is false, there are norm bounded sets $A_1 \subset E_1, \dots, A_n \subset E_n$ and $\varepsilon > 0$ such that for all $y_1 \in E_1^+, \dots, y_n \in E_n^+$, we can find $x_1 \in A_1, \dots, x_n \in A_n$ such that
    $ \norma{A((|x_1| - y_1)^+, \dots, (|x_n| - y_n)^+)} \geq \varepsilon. $
    Fix $x_1^1 \in A_1, \dots, x_n^1 \in A_n$. So, there are $x_1^2 \in A_1, \dots, x_n^2 \in A_n$ such that
    $$ \norma{A((|x_1^2| - 4|x_1^1|)^+, \dots, (|x_n^2| - 4|x_n^1|)^+)} \geq \varepsilon. $$
    By induction, we may construct sequences $(x_j^k)_k \subset A_j$ for every $j = 1, \dots, n$ such that
     \begin{equation}\label{hm3s}\left\|A\left(\left(|x_1^{k+1}| - 4^k \sum_{i=1}^k|x_1^i|\right)^+, \dots, \left(|x_n^{k+1}| - 4^k \sum_{i=1}^k|x_n^i|\right)^+\right)\right\|\geq \varepsilon \text{~for every $k \in \N$}. \end{equation}
    For each $j=1, \dots, n$, define $x_j = \displaystyle \sum_{k=1}^\infty 2^{-k} |x_j^k|$,
    $$ z_j^k = \left ( |x_j^{k+1}| - 4^k \sum_{i=1}^k|x_j^i| \right )^+, \quad \text{and } \quad u_j^k = \left ( |x_j^{k+1}| - 4^k \sum_{i=1}^k|x_j^i| - 2^{-k}x_j \right )^+.  $$
    Thus, for each $j=1, \dots, n$, $(u_j^k)_k$ is a norm bounded disjoint sequence in $E_j$ (see \cite[Lemma 4.35]{alip}) such that
    $ 0 \leq u_j^k \leq z_j^k \leq u_j^k + 2^{-k}x_j$ for every $k \in \N$. Hence
    $$ 0 \leq A(u_1^k, \dots,  u_n^k) \leq A(z_1^k, \dots,  z_n^k) \leq A(u_1^k + 2^{-k}x_1, \dots,  u_n^k + 2^{-k}x_n)  $$
    holds for every $k \in \N$.
    On the one hand, since $A$ is $M$-weakly compact and $(u_1^k)_k, \dots, (u_n^k)_k$ are disjoint sequences, we have $\displaystyle \lim_{k \to \infty} A(u_1^k, \dots,  u_n^k) = 0$. On the other hand, since $\displaystyle \lim_{k \to \infty} 2^{-k} x_j = 0$ and  for every $j=1, \dots, n$, we get
    \begin{align*}
        \lim_{k \to \infty}
        A(u_1^k + 2^{-k}x_1, \dots,  u_n^k + 2^{-k}x_n) & = \lim_{k \to \infty} A(u_1^k, u_2^k + 2^{-k}x_2, \dots,  u_n^k + 2^{-k}x_n) \\
        & = \cdots = \lim_{k \to \infty} A(u_1^k, \dots,  u_n^k) = 0.
    \end{align*}
    Therefore, $\displaystyle \lim_{k \to \infty} A(z_1^k, \dots,  z_n^k) = 0$, which contradicts (\ref{hm3s}).
\end{proof}

\begin{lemma} \label{lemacomp2}
    Let $A: E_1 \times \cdots \times E_n \to F$ be a strongly $M$-weakly compact positive $n$-linear operator. Then, for each $\varepsilon > 0$ there are $z_1 \in E_1^+, \dots, z_n \in E_n^+$ such that
    $$A(B_{E_1} \times \cdots \times B_{E_n}) \subset A([-z_1, z_1] \times \cdots \times [-z_n, z_n]) + \varepsilon B_F.$$
\end{lemma}

\begin{proof} The case $n = 2$ does not capture the main difficulties of the proof, so we prove the case $n =3$, in which the sensitive issues are handled. The argument will make it clear that the general case follows analogously.
    Let $\varepsilon > 0$ be given.
    Since $A : E_1 \times E_2 \times E_3 \to F$ is $M$-weakly compact, applying  Lemma \ref{lemacomp1} for the norm bounded sets $B_{E_1}, B_{E_2}$ and $B_{E_3}$, there are $y_1 \in E_1^+, y_2 \in E_2^+$ and $y_3 \in E_3^+$ such that
    $$ \norma{A((|x_1| - y_1)^+, (|x_2| - y_2)^+, (|x_3| - y_3)^+ )} \leq \frac\varepsilon7   \text{~~for all $x_1 \in B_{E_1}, x_2 \in B_{E_2}, x_3 \in B_{E_3}$}. $$
    Next, we apply Lemma \ref{lemacomp1} to the $M$-weakly compact operators $A(\cdot, \cdot, y_3): E_1 \times E_2 \to F$ and to the norm bounded sets $B_{E_1}$ and $B_{E_2}$ to obtain $u_1 \in E_1^+$ and $u_2 \in E_2^+$ such that
    $$ \norma{A((|x_1| - u_1)^+, (|x_2| - u_2)^+, y_3} \leq \frac\varepsilon7   \text{~~for all $x_1 \in B_{E_1}$ and $x_2 \in B_{E_2}$}. $$
    Analogously, there are $v_1 \in E_1^+$ and $v_3 \in E_3^+$ such that
    $$ \norma{A((|x_1| - v_1)^+, y_2, (|x_3| - v_3)^+} \leq \frac\varepsilon7   \text{~~for all $x_1 \in B_{E_1}$ and $x_3 \in B_{E_3}$}, $$
    and there are $w_2 \in E_2^+$ and $w_3 \in E_3^+$ such that
    $$ \norma{A(y_1, (|x_2| - w_2)^+, (|x_3| - w_3)^+} \leq \frac\varepsilon7   \text{~~for all $x_2 \in B_{E_2}$ and $x_3 \in B_{E_3}$}. $$
     Call $a_1 = y_1 \vee u_1 \vee v_1$, $a_2 = y_2 \vee u_2 \vee w_2$ and $a_3 = y_3 \vee v_3 \vee w_3$. Now, we will apply Lemma \ref{lemacomp1} with $n=1$ for $A(\cdot, a_2, a_3)$, $A(a_1, \cdot, a_3)$ and $A(a_1, a_2, \cdot)$ with respect to the norm bounded sets $B_{E_1}, B_{E_2}$ and $B_{E_3}$. Thus, there are $b_1 \in E_1^+$, $b_2 \in E_2^+$ and $b_3 \in E_3^+$ so that
    $$ \norma{A((|x_1| - b_1)^+, a_2, a_3} \leq \frac\varepsilon{21} \text{~~for all $x_1 \in B_{E_1}$}, $$
$$ \norma{A(a_1, (|x_2| - b_2)^+, a_3} \leq \frac\varepsilon{21}   \text{~~for all $x_2 \in B_{E_2}$}, $$
    and
    $$ \norma{A(a_1, a_2, (|x_3| - b_3)^+} \leq \frac\varepsilon{21}   \text{~~for all $x_3 \in B_{E_3}$}.$$
Define $z_1 = 13 a_1 \vee b_1$, $z_2 = a_2 \vee b_2$ and $z_3 = a_3 \vee b_3$.
    Let $x_1 \in B_{E_1}$, $x_2 \in B_{E_2}$ and $x_3 \in B_{E_3}$ be given. Using \cite[Theorem 1.7(1)]{alip}, the positivity of $A$ and the linearity of $A$ in each variable, we have
\begin{align*}
     A(x_1, & x_2, x_3)  \leq A(|x_1|, |x_2|, |x_3|) \\&  = A((|x_1| - y_1)^+ + |x_1| \wedge y_1, (|x_2| - y_2)^+ + |x_2| \wedge y_2, (|x_3| - y_3)^+ + |x_3| \wedge y_3) \\
     & \leq  A((|x_1| - y_1)^+, (|x_2| - y_2)^+, (|x_3| - y_3)^+) + A(|x_1|, |x_2|, y_3) + A(|x_1|, y_2, |x_3|) \\
     & + A(|x_1|, y_2, y_3) + A(y_1, |x_2|, |x_3|) + A(y_1,|x_2|, y_3) + A(y_1, y_2, |x_3|) + A(y_1, y_2, y_3).
\end{align*}
Let us investigate the terms $A(|x_1|, |x_2|, y_3)$, $A(|x_1|, y_2, |x_3|)$ and
$A(y_1, |x_2|, |x_3|)$ separately. By applying \cite[Theorem 1.7(1)]{alip}, and (again) the positivity of $A$ and the linearity  of $A$ in each variable of $A$, we have
\begin{align*}
    A(|x_1|, & |x_2|, y_3)  = A((|x_1| - u_1)^+ + |x_1| \wedge u_1, (|x_2| - u_2)^+ + |x_2| \wedge u_2, y_3) \\
    & \leq A((|x_1| - u_1)^+, (|x_2| - u_2)^+, y_3) + A(|x_1|, u_2, y_3) + A(u_1, |x_2|, y_3) + A(u_1, u_2, y_3) \\
    & \leq A((|x_1| - u_1)^+, (|x_2| - u_2)^+, a_3) + A(|x_1|, a_2, a_3) + A(a_1, |x_2|, a_3) + A(a_1, a_2, a_3).
\end{align*}
Analogously,
\begin{align*}
    A(|x_1|, & y_2, |x_3|)  = A((|x_1| - v_1)^+ + |x_1| \wedge v_1, y_2, (|x_3| - v_3)^+ + |x_3| \wedge v_3) \\
    & \leq A((|x_1| - v_1)^+, y_2, (|x_3| - v_3)^+) + A(|x_1|, y_2, v_3) + A(v_1, y_2, |x_3|) + A(v_1, y_2, v_3) \\
    & \leq A((|x_1| - v_1)^+, y_2, (|x_3| - v_3)^+) + A(|x_1|, a_2, a_3) + A(a_1, a_2, |x_3|) + A(a_1, a_2, a_3),
\end{align*}
and
\begin{align*}
    A(y_1, & |x_2|, |x_3|)  = A(y_1, (|x_2| - w_2)^+ + |x_2| \wedge w_2,(|x_3| - w_3)^+ + |x_3| \wedge w_3) \\
    & \leq A(y_1, (|x_2| - w_2)^+, (|x_3| - w_3)^+) + A(y_1, |x_2|, w_3) + A(y_1, w_2, |x_3|) + A(y_1, w_2, w_3) \\
    & \leq A(y_1, (|x_2| - w_2)^+, (|x_3| - w_3)^+) + A(a_1, |x_2|, a_3) + A(a_1, a_2, |x_3|) + A(a_1, a_2, a_3).
\end{align*}
    Combining the information above, we get
    \begin{align*}
      A(x_1, & x_2, x_3) \leq  A((|x_1| - y_1)^+, (|x_2| - y_2)^+, (|x_3| - y_3)^+) + A(|x_1|, |x_2|, y_3) + A(|x_1|, y_2, |x_3|) \\
     & + A(|x_1|, y_2, y_3) + A(y_1, |x_2|, |x_3|) + A(y_1,|x_2|, y_3) + A(y_1, y_2, |x_3|) + A(y_1, y_2, y_3) \\
     & \leq A((|x_1| - y_1)^+, (|x_2| - y_2)^+, (|x_3| - y_3)^+) + \\
     & + A((|x_1| - u_1)^+, (|x_2| - u_2)^+, a_3) + A(|x_1|, a_2, a_3) + A(a_1, |x_2|, a_3) + A(a_1, a_2, a_3) \\
     & + A((|x_1| - v_1)^+, y_2, (|x_3| - v_3)^+) + A(|x_1|, a_2, a_3) + A(a_1, a_2, |x_3|) + A(a_1, a_2, a_3) \\
     & + A(|x_1|, a_2, a_3) \\
     & + A(y_1, (|x_2| - w_2)^+, (|x_3| - w_3)^+) + A(a_1, |x_2|, a_3) + A(a_1, a_2, |x_3|) + A(a_1, a_2, a_3) \\
     & + A(a_1, |x_2|, a_3) + A(a_1, a_2, |x_3|) + A(a_1, a_2, a_3),
    \end{align*}
    that is
\begin{align*}
      A(x_1, & x_2, x_3) \leq A((|x_1| - y_1)^+, (|x_2| - y_2)^+, (|x_3| - y_3)^+) + A((|x_1| - u_1)^+, (|x_2| - u_2)^+, a_3) \\
      & + A((|x_1| - v_1)^+, y_2, (|x_3| - v_3)^+) + A(y_1, (|x_2| - w_2)^+, (|x_3| - w_3)^+) \\
     & + 3 A(|x_1|, a_2, a_3) + 3A(a_1, |x_2|, a_3) + 3 A(a_1, a_2, |x_3|) + 4A(a_1, a_2, a_3).
    \end{align*}
Now we handle the terms $A(|x_1|, a_2, a_3)$, $A(a_1, |x_2|, a_3)$ and $A(a_1, a_2, |x_3|)$ separately. Using once again \cite[Theorem 1.7(1)]{alip}, the positivity of $A$ and its linearity in each variable, we have
\begin{align*}
    A(|x_1|, a_2, a_3) & = A((|x_1| - b_1)^+, a_2, a_3) + A(|x_1|\wedge b_1, a_2, a_3) \\
    & \leq A((|x_1| - b_1)^+, a_2, a_3) + A(b_1, a_2, a_3) \\
    & \leq A((|x_1| - b_1)^+, a_2, a_3) + A(a_1 \vee b_1, a_2 \vee b_2, a_3 \vee b_3).
\end{align*}
Analogously
\begin{align*}
    A(a_1, |x_2|, a_3) & \leq A(a_1, (|x_2| - b_2)^+, a_3) + A(a_1 \vee b_1, a_2 \vee b_2, a_3 \vee b_3),
\end{align*}
and
\begin{align*}
    A(a_1, a_2, |x_3|) & \leq A(a_1, a_2, (|x_3| - b_3)^+) + A(a_1 \vee b_1, a_2 \vee b_2, a_3 \vee b_3).
\end{align*}
Combining the last four inequalities above, we obtain
\begin{align*}
      A(x_1, & x_2, x_3) \leq A((|x_1| - y_1)^+, (|x_2| - y_2)^+, (|x_3| - y_3)^+) + A((|x_1| - u_1)^+, (|x_2| - u_2)^+, a_3) \\
      & + A((|x_1| - v_1)^+, y_2, (|x_3| - v_3)^+) + A(y_1, (|x_2| - w_2)^+, (|x_3| - w_3)^+) \\
     & + 3 A((|x_1| - b_1)^+, a_2, a_3) + 3A(a_1 \vee b_1, a_2 \vee b_2, a_3 \vee b_3) \\
     & + 3A(a_1, (|x_2| - b_2)^+, a_3) + 3 A(a_1 \vee b_1, a_2 \vee b_2, a_3 \vee b_3) \\
     & + 3A(a_1, a_2, (|x_3| - b_3)^+) + 3A(a_1 \vee b_1, a_2 \vee b_2, a_3 \vee b_3)
     + 4A(a_1, a_2, a_3),
    \end{align*}
that is
\begin{align*}
     A(x_1, & x_2, x_3) \leq A((|x_1| - y_1)^+, (|x_2| - y_2)^+, (|x_3| - y_3)^+) + A((|x_1| - u_1)^+, (|x_2| - u_2)^+, a_3) \\
      & + A((|x_1| - v_1)^+, y_2, (|x_3| - v_3)^+) + A(y_1, (|x_2| - w_2)^+, (|x_3| - w_3)^+) \\
      & + 3 A((|x_1| - b_1)^+, a_2, a_3) + 3A(a_1, (|x_2| - b_2)^+, a_3) + 3 A(a_1, a_2, (|x_3| - b_3)^+) \\
      & + 13 A(a_1 \vee b_1, a_2 \vee b_2, a_3 \vee b_3).
\end{align*}
Recalling that $z_1 = 13 a_1 \vee b_1$, $z_2 = a_2 \vee b_2$ and $z_3 = a_3 \vee b_3$, from the inequality above together with the norm estimates obtained at the beginning of the proof, we have
\begin{align*}
    \norma{A(x_1 & , x_2, x_3) - A(z_1, z_2, z_3)}  \leq \frac{\varepsilon}{7} + \frac{\varepsilon}{7} + \frac{\varepsilon}{7} + \frac{\varepsilon}{7}  + 3\frac{\varepsilon}{21} + 3 \frac{\varepsilon}{21} + 3 \frac{\varepsilon}{21} = \varepsilon.
\end{align*}
    Therefore $A(x_1, x_2, x_3) - A(z_1, z_2, z_3) \in \varepsilon B_F$, and we are done.
\end{proof}

Now, we have all we need to present the proof of Theorem \ref{corcomp}.

\medskip

\noindent {\it Proof of Theorem \ref{corcomp}.} We shall use (twice) that a subset $K$ of a Banach space $X$ is relatively compact if and only if for every $\varepsilon > 0$ there is a relatively compact set $K_\varepsilon$ in $X$ such that $K \subset K_\varepsilon + \varepsilon B_X$ (see, e.g., \cite[p. 5]{diestelss}). By assumption, $A: E_1 \times \cdots \times E_n \to F$ is a positive strongly $M$-weakly compact $n$-linear operator. Let $\varepsilon > 0$ be given. By Lemma \ref{lemacomp2} there are $y_1 \in E_1^+, \dots, y_n \in E_n^+$ such that
\begin{equation}\label{t8m2}A(B_{E_1} \times \cdots \times B_{E_n}) \subset A([-y_1,y_1]\times \cdots \times [-y_n, y_n]) + \varepsilon \, B_F.  \end{equation}

Suppose that $E_1, \dots, E_m$ are atomic with order continuous norms. In this case, the order intervals $[-y_1, y_1], \dots, [-y_n, y_n]$ are relatively compact in $E_1, \dots, E_n$, respectively (see \cite[Theorem 6.1]{WOrder}). So,  $[-y_1,y_1]\times \cdots \times [-y_n, y_n]$ is relatively compact in $E_1 \times \cdots \times E_n$, and the continuity of $A$ yields that $A([-y_1,y_1]\times \cdots \times [-y_n, y_n])$ is relatively compact in $F$. Together with (\ref{t8m2}), this proves that $A(B_{E_1} \times \cdots \times B_{E_n})$ is relatively compact, hence $A$ is a compact operator.

 Assume now that $F$ is atomic with order continuous norm. Since $A$ is positive, $A$ is order bounded, so there exists $z \in F$ such that $A([-y_1,y_1]\times \cdots \times [-y_n, y_n]) \subset [-z,z]$. By (\ref{t8m2}) we have
$$ A(B_{E_1} \times \cdots \times B_{E_n}) \subset [-z,z] + \varepsilon \, B_F.  $$
Finally, it follows from \cite[Theorem 6.1]{WOrder} that $[-z,z]$ is relatively compact in $F$, hence $A$ is a compact operator. \qed

\medskip

Now our main result follows from a combination of Theorem \ref{corcomp} and Examples \ref{exemplos1}:

%

\begin{theorem}  \label{exemplos2} Let $1 < p_1, \ldots, p_n < \infty, 1\leq q < \infty$ be given and let $\mu_1, \ldots, \mu_n, \nu$ be measures.\\
   {\rm (1)} All positive $n$-linear operators from $L_{p_1}(\mu_1) \times \cdots \times L_{p_n}(\mu_n)$ to $\ell_q$ and from $\ell_{p_1} \times \cdots \times \ell_{p_n}$ to $L_q(\nu)$ are compact whenever $\sum\limits_{i=1}^n \frac{1}{p_i} < \frac{1}{q}$. \\
   {\rm (2)} All positive $n$-linear operators from ${\rm FBL}[L_{p_1}(\mu_1)] \times \cdots \times {\rm FBL}[L_{p_n}(\mu_n)]$ to $\ell_q $
are compact  whenever $1 < p_1, \dots, p_n < 2$, $1 \leq q < \infty$ and
$\sum\limits_{i=1}^n \frac{1}{p_i} < \frac{1}{q}$.\\
{\rm (3)} All positive $n$-linear operators from $\ell_{p_1} \times \cdots \times \ell_{p_n}$ to $({\rm FBL}[L_{q}(\nu)])^* $
are compact whenever $2 \leq q < \infty$ and $\sum\limits_{i=1}^n \frac{1}{p_i} < \dfrac{1}{\sigma(({\rm FBL}[L_{2}(\nu)])^*)} = \frac{1}{2}$. The same holds if we replace $L_q(\mu)$ with a Banach lattice $F$ such that $2 \leq s(F) < \infty$.
\end{theorem}

We conclude our manuscript with two applications of Theorem \ref{exemplos2}. Recall that a $n$-homogeneous polynomial $P: E \to F$ between Banach lattices is said to be positive if its associated symmetric $n$-linear operator $T_P: E^n \to F$ is positive. A homogeneous polynomial is regular if it is the difference of two positive polynomials. By $\mathcal{P}^r(^n E; F)$ we denote the space of regular $n$-homogeneous polynomials from $E$ to $F$. Details can be found in \cite{bubuskes, loane}.

\begin{corollary} Let $n \in \N$ and $1 \leq p,q < \infty$ be such that $q < np$, and let $\mu$ be a measure. Then, every positive $n$-homogeneous polynomial $P: L_p(\mu) \to \ell_q$ is compact, that is, $P(B_{L_p(\mu)})$ is a relatively compact subset of $\ell_q$. In this case, $\mathcal{P}^r(^n L_p(\mu); \ell_q)$ does not contain a copy of $c_0$.
\end{corollary}

\begin{proof} The symmetric $n$-linear operator $T_P$ associated to $P$ is compact by Theorem \ref{exemplos2}(1). Since $P(B_{L_p(\mu)})\subset T_P\left( \left(B_{L_p(\mu)}\right)^n\right)$, $P$ is compact as well. The second statement follows from \cite[Theorem 4.3]{botmirru}.
\end{proof}


\begin{corollary} Let $2 < p < \infty$, $2 \leq q < \infty$ be given and let $\mu$ be a measure. Then, every positive linear operator from $\ell_p$ to ${(\rm FBL}[L_q(\mu)])^*$ is compact and is norm-attaining. The same holds if we replace $L_q(\mu)$ with a Banach lattice $F$ such that $2 \leq s(F) < \infty$.
\end{corollary}

\begin{proof} In this case, $\ell_p$ is a reflexive Banach lattice whose order is given by a basis, so every positive linear operator from $\ell_p$ to ${(\rm FBL}[L_q(\mu)])^*$ is compact by Theorem \ref{exemplos2}(3). The second statement follows from \cite[Theorem 2.12]{luizmiranda}.
\end{proof}

\medskip

\noindent \textbf{Acknowledgement.} Part of this note was prepared while the second author was visiting the Instituto de Matemática e Estatística at the Universidade Federal de Uberlândia in August 2025. He is deeply grateful to the intitute for its hospitality and support.

\noindent G. Botelho\\
Instituto de Matemática e Estatística\\
Universidade Federal de Uberlândia\\
38.400-902 -- Uberl\^andia -- Brazil\\
e-mail: botelho@ufu.br

\medskip

\noindent V. C. C. Miranda\\
Centro de Matem\'atica, Computa\c c\~ao e Cogni\c c\~ao \\
Universidade Federal do ABC \\
09.210-580 -- Santo Andr\'e -- 		Brazil.  \\
e-mail: colferaiv@gmail.com

\end{document}